\documentclass{article}[11pt]
\usepackage{amssymb}
\usepackage{amsmath}
\usepackage{amsthm}
\newtheorem{theorem}{Theorem}[section]
\newtheorem{lemma}[theorem]{Lemma}
\newtheorem{definition}[theorem]{Definition}

\newtheorem{proposition}[theorem]{Proposition}
\newtheorem{remark}[theorem]{Remark}
\newtheorem{examples}[theorem]{Examples}

\newcommand{\R}{{\mathbb{R}}}

\begin{document}

\title{Generalized Orlicz Spaces and Wasserstein Distances for Convex-Concave Scale Functions}

\author{Karl-Theodor Sturm}

\date{}

\maketitle

\begin{abstract}
Given a strictly increasing, continuous function $\vartheta:\R_+\to\R_+$, based on the cost functional
$\int_{X\times X}\vartheta\left(d(x,y)\right)\,d q(x,y)$, we define the $L^\vartheta$-Wasserstein distance $W_\vartheta(\mu,\nu)$ between probability measures $\mu,\nu$ on some metric space $(X,d)$.
The function $\vartheta$ will be assumed to admit a representation $\vartheta=\varphi\circ\psi$ as a composition of a convex and a concave function $\varphi$ and $\psi$, resp.
Besides convex functions and concave functions this includes all $\mathcal C^2$ functions.

For such functions $\vartheta$ we extend the concept of Orlicz spaces, defining the metric space $L^\vartheta(X,m)$ of measurable functions $f: X\to\R$ such that, for instance,
$$d_\vartheta(f,g)\le1\quad\Longleftrightarrow\quad \int_X\vartheta(|f(x)-g(x)|)\,d\mu(x)\le1.$$
\end{abstract}
%\section{$L^\vartheta$-Wasserstein Distance for CCC Functions $\vartheta$}

\section{Convex-Concave Compositions}

Throughout this paper, $\vartheta$ will be a strictly increasing, continuous function from $\mathbb R_+$ to $\mathbb R_+$ with $\vartheta(0)=0$.

\begin{definition}
$\vartheta$ will be called ccc function ("convex-concave composition") iff there exist two
 strictly increasing continuous functions $\varphi,\psi: \mathbb R_+\to\mathbb R_+$ with
$\varphi(0)=\psi(0)=0$
s.t.  $\varphi$ is convex, $\psi$ is concave and
$$\vartheta=\varphi\circ\psi.$$
The pair $(\varphi,\psi)$ will be called convex-concave factorization of $\vartheta$.

The factorization is called minimal (or non-redundant) if for any other factorization $(\tilde\varphi,\tilde\psi)$ the function $\varphi^{-1}\circ \tilde\varphi$ is convex.
\end{definition}

Two minimal factorizations of a given function $\vartheta$ differ only by a linear change of variables. Indeed, if $\varphi^{-1}\circ\tilde\varphi$ is convex and also $\tilde\varphi^{-1}\circ\varphi$ is convex then there exists a $\lambda\in(0,\infty)$ s.t.  $\tilde\varphi(t)=\varphi(\lambda t)$ and  $\tilde\psi(t)=\frac1\lambda\psi(t)$.

\medskip

For each convex, concave or ccc function $f:\mathbb R_+\to\mathbb R_+$ put $f'(t):=f'(t+):=\lim_{h\searrow0}\frac1h\left[f(t+h)-f(t)\right]$.

\begin{lemma}
(i) For any ccc function $\vartheta$, the function $\log\vartheta'$ is locally of bounded variation and the distribution $(\log\vartheta')'$ defines a signed Radon measure on $(0,\infty)$, henceforth denoted by $d(\log\vartheta')$.

(ii)
A pair $(\varphi,\psi)$ of strictly increasing convex or concave, resp., continuous functions with $\varphi(0)=\psi(0)=0$
is a factorization of $\vartheta$ iff
\begin{equation}\label{radon}
d(\log\vartheta')=\psi^{-1}_*d(\log\varphi')+d(\log\psi')
\end{equation}
in the sense of signed Radon measures.

(iii)
The factorization $(\varphi,\psi)$ is minimal iff
for any other factorization $(\tilde\varphi,\tilde\psi)$
$$-d(\log\psi')\le -d(\log\tilde\psi')$$
in the sense of nonnegative Radon measures on $(0,\infty)$.

(iv) Every ccc function $\vartheta$ admits a minimal factorization $(\check\vartheta,\hat\vartheta)$ given by
$\check\vartheta:=\vartheta\circ\hat\vartheta^{-1}$ and
$$\hat\vartheta(x):= \int_0^x\exp\left(-\int_1^y d\nu_-(z)\right)dy$$
where $d\nu_- (z)$ denotes the negative part of the Radon measure $d\nu(z)=d(\log\vartheta')(z)$.
\end{lemma}

\begin{proof} (i), (ii):\
The chain rule for convex/concave functions yields
$$\vartheta'(t)=\varphi'(\psi(t))\cdot\psi'(t)$$
for each factorization $(\varphi,\psi)$ of a ccc function $\vartheta$.
Taking logarithms it implies that $\log\vartheta'$ locally is a BV function (as a difference of two increasing functions) and, hence, that the associated Radon measures satisfy
\begin{eqnarray*}
d(\log\vartheta')&=&d(\log\varphi'\circ\psi)+d(\log\psi')\\
&=&\psi^{-1}_*d(\log\varphi')+d(\log\psi').
\end{eqnarray*}

(iii): \
The factorization $(\varphi,\psi)$ is minimal if and only if for any other factorization $(\tilde\varphi,\tilde\psi)$ the function $u=\varphi^{-1}\circ\varphi=\psi\circ\tilde\psi^{-1}$ is convex.
Since $\log\psi'=\log u'(\tilde\psi)+\log\tilde\psi'$, the latter is equivalent to
$$d(\log\psi')\ge d(\log\tilde\psi')$$
which is the claim.

(iv): \
Define $\hat\vartheta$ as above. It remains to verify that $\hat\vartheta<\infty$.
Let $(\varphi,\psi)$ be any convex-concave factorization of $\vartheta$. Without restriction assume $\psi'(1)=1$.
Then the Hahn decomposition of (\ref{radon}) yields
\begin{equation}\label{hahn}
d\nu_- \le -d(\log\psi').
\end{equation}
Hence, for all $0\le x \le 1$
\begin{eqnarray*}
0\le\hat\vartheta(x)&=&
\int_0^x\exp\left(\int_y^1 d\nu_- (z)\right)dy\\
&\le&
\int_0^x\exp\left(-\int_y^1 d(\log\psi')(z)\right)dy
=\psi(x)<\infty.
\end{eqnarray*}
This already implies that $\hat\vartheta$ is finite, strictly increasing and continuous on $[0,\infty)$. (For instance, for $x>1$ it follows $\hat\vartheta(x)\le \hat\vartheta(1)+x-1$.) Moreover, one easily verifies that $\hat\vartheta$ is concave.

Since $\nu_+, \nu_-$ are the minimal nonnegative measures in the ('Hahn' or 'Jordan') decomposition of $\nu=\nu_+-\nu_-$, it follows that $(\check\vartheta,\hat\vartheta)$ is a minimal cc decomposition of $\vartheta$.
\end{proof}

\begin{examples}
\begin{itemize}
\item Each convex function $\vartheta$  is a ccc function. A minimal factorization is given by $(\vartheta,Id)$.
\item Each concave function $\vartheta$  is a ccc function. A minimal factorization is given by $(Id,\vartheta)$.
\item Each $\mathcal C^2$ function $\vartheta$ with $\vartheta'(0+)>0$ is a ccc function. The minimal factorization
is given by
$$\hat\vartheta(x):=\int_0^x\exp\left(\int_1^y\frac{\vartheta''(z)\wedge 0}{\vartheta'(z)}dz\right)dy$$
and $\check\vartheta:=\vartheta\circ\hat\vartheta^{-1}$.
(The condition $\vartheta'(0+)>0$ can be replaced by the strictly weaker requirement that the previous integral defining $\hat\vartheta$ is finite.)
\end{itemize}
\end{examples}

\section{The Metric Space $L^\vartheta(X,\mu)$}

Let $(X,\Xi, \mu)$ be a $\sigma$-finite measure space and $(\varphi,\psi)$ a minimal ccc factorization of a given function $\vartheta$.
Then $L^\vartheta(X,\mu)$ will denote the space of all
measurable functions $f : X\to\mathbb R$ such that
$$\int_X \varphi\left(\frac1t\psi(|f|)\right)\, d\mu < \infty$$
for some $t\in(0,\infty)$
where as usual functions which agree almost everywhere are identified.
Note that -- due to the fact that $r\mapsto\varphi(r)$ for large $r$ grows at least linearly -- the previous condition is equivalent to the condition
$\int_X \varphi\left(\frac1t\psi(|f|)\right)\, d\mu \le1$
for some $t\in(0,\infty)$.

\begin{theorem}
$L^\vartheta(X,\mu)$ is a complete metric space with the metric
$$d_\vartheta(f,g)=\inf\left\{t\in(0,\infty): \ \int_X \varphi\left(\frac1t\psi(|f-g|)\right)\, d\mu\le1\right\}.$$
\end{theorem}

The definition of this metric does not depend on the choice of the minimal ccc factorization of the function $\vartheta$. However, choosing an arbitrary convex-concave factorization of $\vartheta$ might change the value of $d_\vartheta$.

Note that always $d_\vartheta(f,g)=d_\vartheta(f-g,0)$.

\begin{proof}
Let $f,g,h\in L^\vartheta(X,\mu)$ be given and choose $r,s>0$ with $d_\vartheta(f,g)<r$ and $d_\vartheta(g,h)<s$. The latter implies
$$\int_X \varphi\left(\frac1r\psi(|f-g|)\right)\, d\mu\le1, \qquad
\int_X \varphi\left(\frac1s\psi(|g-h|)\right)\, d\mu\le1.$$
Concavity of $\psi$ yields $\psi(|f-h|)\le\psi(|f-g|)+\psi(|g-h|)$.
Put $t=r+s$. Then convexity of $\varphi$ implies
$$\varphi\left(\frac1t\psi(|f-h|)\right)\le\varphi\left(\frac rt\cdot \frac{\psi(|f-g|)}r+\frac st\cdot \frac{\psi(|g-h|)}s\right)
\le\frac rt\cdot \varphi\left(\frac{\psi(|f-g|)}r\right)+\frac st\cdot \varphi\left(\frac{\psi(|g-h|)}s\right).$$
Hence,
$$\int_X\varphi\left(\frac1t\psi(|f-h|)\right)d\mu\le
\frac rt\cdot \int_X\varphi\left(\frac{\psi(|f-g|)}r\right)d\mu+\frac st\cdot\int_X \varphi\left(\frac{\psi(|g-h|)}s\right)d\mu\le\frac rt\cdot1+\frac st\cdot1=1$$
and thus
$d_\vartheta(f,h)\le t$. This proves that $d_\vartheta(f,h)\le d_\vartheta(f,g)+ d_\vartheta(g,h)$.

In order to prove the completeness of the metric, let $(f_n)_n$ be a Cauchy sequence in $L^\vartheta$. Then $d_\vartheta(f_n,f_m)< \epsilon_n$ for all $n,m$ with $m\ge n$ and suitable $\epsilon_n\searrow0$. Choose  an increasing sequence of measurable sets $X_k$, $k\in\mathbb N$, with $\mu(X_k)<\infty$ and $\cup_k X_k=X$.
Then
$$\int_{X_k} \varphi\left(\frac1{\epsilon_n}\psi(|f_n-f_m|)\right)\, d\mu \le1$$
for all $k,m,n$ with $m\ge n$.
Jensen's inequality implies
$$\varphi\left(\frac1{\mu(X_k)}\int_{X_k} \frac1{\epsilon_n}\psi(|f_n-f_m|)\,d\mu\right) \le\frac1{\mu(X_k)}$$
and thus
$$\int_{X_k} \left|\psi(f_n)-\psi(f_m)\right|\,d\mu \le\epsilon_n\cdot \mu(X_k)\cdot\varphi^{-1}\left(\frac1{\mu(X_k)}\right).$$
In other words,
$(\psi(f_n))_n$ is a Cauchy sequence in $L^1(X_k,\mu)$. It follows that it has a subsequence $(\psi({f_n}_i))_i$ which converges $\mu$-almost everywhere on $X_k$.
In particular, $({f_n}_i)_i$ converges almost everywhere on $X_k$ towards some limiting function $f$ (which easily is shown to be independent of $k$).

Finally, Fatou's lemma now implies
$$\int_{X_k} \varphi\left(\frac1{\epsilon_n}\psi(|f_n-f|)\right)\, d\mu \le\liminf_{m\to\infty}\int_{X_k} \varphi\left(\frac1{\epsilon_n}\psi(|f_n-f_m|)\right)\, d\mu \le1$$
for each $k$ and $n\in\mathbb N$. Hence,
$$\int_{X} \varphi\left(\frac1{\epsilon_n}\psi(|f_n-f|)\right)\, d\mu \le1,$$
that is,
$$d_\vartheta(f_n,f)\le\epsilon_n$$
which proves the claim.

Finally, it remains to verify that
$$d_\vartheta(f,g)=0\quad\Longleftrightarrow\quad f=g\ \mu\mbox{-a.e. on }X.$$
The implication $\Leftarrow$ is trivial. For the reverse implication, we may argue as in the previous completeness proof: $d_\vartheta(f,g)=0$ will yield
$\int_{X_k} \varphi\left(\frac1{t}\psi(|f-g|)\right)\, d\mu \le1$ for all $k\in\mathbb N$ and all $t>0$ which in turn implies
$\int_{X_k} \left|\psi(f)-\psi(g)\right|\,d\mu =0$.
The latter proves $f=g$ $\mu$-a.e. on $X$ which is the claim.
\end{proof}

\begin{examples} If $\vartheta(r)=r^p$ for some $p\in(0,\infty)$ then
$$d_\vartheta(f,g)=\left(\int_X |f-g|^p\,d\mu\right)^{1/p^*}$$
with $p^*:=p$ if $p\ge 1$ and $p^*:=1$ if $p\le1$.
\end{examples}

\begin{proposition}
\begin{itemize}
\item[(i)] If $\vartheta$ is convex then $\|f\|_{L^\vartheta(X,\mu)}:=d_\vartheta(f,0)$ is indeed a norm and $L^\vartheta(X,\mu)$ is a Banach space, called  Orlicz space. The norm is called Luxemburg norm.

\item[(ii)] If $\vartheta$ is concave then $$d_\vartheta(f,g)=\int_X \vartheta(|f-g|)\,d\mu\ge \|\vartheta(f)-\vartheta(g)\|_{L^1(X,\mu)}.$$

 \item[(iii)] For general ccc function $\vartheta=\varphi\circ\psi$
    $$d_\vartheta(f,g)=\|\psi(|f-g|)\|_{L^\varphi(X,\mu)}.$$
\item[(iv)] If $\mu(M)=1$ then for each strictly increasing, convex function $\Phi:\mathbb R_+\to\mathbb R_+$ with $\Phi^{-1}(1)=1$
$$d_{\Phi\circ\vartheta}(f,g)\ge d_\vartheta(f,g)$$
("Jensen's inequality").
\end{itemize}
\end{proposition}

\begin{proof}
(i) If $\psi(r)=cr$ then obviously $d_\vartheta(tf,0)=t\cdot d_\vartheta(f,0)$. See also standard literature \cite{Orlicz}.

(ii) Concavity of $\vartheta$ implies $\vartheta(|f-g|)\ge |\vartheta(f)-\vartheta(g)|$.

(iv) Assume that $d_{\Phi\circ\vartheta}(f,g)< t$ for some $t\in(0,\infty)$. It implies
$$\int_X \Phi\left(\varphi\left(\frac1t\psi(|f-g|)\right)\right)\, d\mu \le1.$$
Classical Jensen inequality for integrals yields
$$\Phi\left(\int_X \varphi\left(\frac1t\psi(|f-g|)\right)\, d\mu\right) \le1$$
which -- due to the fact that $\Phi^{-1}(1)=1$ -- in turn implies
$d_\vartheta(f,g)\le t$.
\end{proof}

\section{The $L^\vartheta$-Wasserstein Space}

Let $(X,d)$ be a complete separable metric space and $\vartheta$ a ccc function with minimal factorization $(\varphi,\psi)$.
The \emph{$L^\vartheta$-Wasserstein space} $\mathcal P_\vartheta(X)$ is defined as the space of all probability measures $\mu$ on $X$ -- equipped with its Borel $\sigma$-field --
s.t.
$$\int_X \varphi\left(\frac1t\psi(d(x,y))\right)\, d\mu(x) < \infty$$
for some $y\in X$ and some $t\in(0,\infty)$. The $L^\vartheta$-Wasserstein distance of two probability measures $\mu,\nu\in \mathcal P_\vartheta(X)$ is defined as
\begin{eqnarray*}
W_\vartheta(\mu,\nu)&=&\inf\left\{t>0: \ \inf_{q\in\Pi(\mu,\nu)}\int_{X\times X} \varphi\left(\frac1t\psi(d(x,y))\right)\, dq(x,y)\le 1\right\}
\end{eqnarray*}
where $\Pi(\mu,\nu)$ denotes the set of all couplings of $\mu$ and $\nu$, i.e. the set of all probability measures $q$ on $X\times X$ s.t. $q(A\times X)=\mu(A), q(X\times A)=\nu(A)$ for all Borel sets $A\subset X$.

Given two probability measures $\mu,\nu\in \mathcal P_\vartheta(X)$, a coupling $q$ of them is called \emph{optimal} iff
$$\int_{X\times X} \varphi\left(\frac1w\psi(d(x,y))\right)\, dq(x,y)\le 1$$
for $w:=W_\vartheta(\mu,\nu)$.
\begin{proposition}
For each pair of probability measures $\mu,\nu\in \mathcal P_\vartheta(X)$ there exists an optimal coupling $q$.
\end{proposition}

\begin{proof} For $t\in(0,\infty)$ define the cost function $c_t(x,y)=\varphi(\frac1t\psi(d(x,y)))$. Note that $t\mapsto c_t(x,y)$ is continuous and decreasing.

Given $\mu,\nu$ s.t. $w:=W_\vartheta(\mu,\nu)<\infty$. Then for all $t>w$ the measures $\mu$ and $\nu$ have finite $c_t$-transportation costs. More precisely,
$$\inf_{q\in\Pi(\mu,\nu)}\int_{X\times X} c_t(x,y)\, dq(x,y)\le 1.$$
Hence, there exists $q_n\in\Pi(\mu,\nu)$ s.t.
$$\int_{X\times X} c_{w+\frac1n}(x,y)\, dq_n(x,y)\le 1+\frac1n.$$
In particular, $\int_{X\times X} c_{w+1}(x,y))\, dq_n(x,y)\le 2$
for all $n\in\mathbb N$.
Hence, the family $(q_n)_n$ is tight (\cite{Villani2}, Lemma 4.4).
Therefore, there exists a  converging subsequence $\left({q_n}_k\right)_k$
with limit $q\in\Pi(\mu,\nu)$
satisfying
$$\int_{X\times X} c_{w+\frac1n}(x,y)\, dq(x,y)\le 1+\frac1n$$
for all $n$ (\cite{Villani2}, Lemma 4.3) and thus
$$\int_{X\times X} c_{w}(x,y)\, dq(x,y)\le 1.$$
\end{proof}

\begin{proposition}
$W_\vartheta$ is a complete metric on $\mathcal P_\vartheta(X)$.
\end{proposition}
The triangle inequality for $W_\vartheta$ is valid not only on $\mathcal P_\vartheta(X)$ but on the whole space $\mathcal P(X)$ of probability measures on $X$. The triangle inequality implies that
$W_\vartheta(\mu,\nu)<\infty$ for all $\mu,\nu\in\mathcal P_\vartheta(X)$.

\begin{proof}
Given three probability measures $\mu_1,\mu_2,\mu_3$ on $X$ and numbers $r,s$ with
$W_\vartheta(\mu_1,\mu_2)<r$ and $W_\vartheta(\mu_2,\mu_3)<s$.
Then there exist a coupling $q_{12}$ of $\mu_1$ and $\mu_2$ and a coupling $q_{23}$ of $\mu_2$ and $\mu_3$ s.t.
$$\int\varphi\left(\frac1r\psi\circ d\right)\,dq_{12}\le1,\qquad \int\varphi\left(\frac1s\psi\circ d\right)\,dq_{23}\le1.$$
Let $q_{123}$ be the gluing of the two couplings $q_{12}$ and $q_{23}$, see e.g. \cite{Dudley}, Lemma 11.8.3.
That is, $q_{123}$ is a probability measure on $X\times X\times X$ s.t. the projection onto the first two factors coincides with $q_{12}$ and the projection onto the last two factors coincides with $q_{23}$.
Let $q_{13}$ denote the projection of $q_{123}$ onto the first and third factor. In particular, this will be a coupling of $\mu_1$ and $\mu_3$. Then for $t:=r+s$
\begin{eqnarray*}
\lefteqn{\int_{X\times X}\varphi\left(\frac1t\psi(d(x,z))\right)\,dq_{13}(x,z)}\\
&\le&
\int_{X\times X\times X}\varphi\left(\frac1t\psi(d(x,y)+d(y,z))\right)\,dq_{123}(x,y,z)\\
&\le&
\int_{X\times X\times X}\varphi\left(\frac rt\frac{\psi(d(x,y))}r+\frac st\frac{\psi(d(y,z))}s\right)\,dq_{123}(x,y,z)\\
&\le&
\frac rt\int_{X\times X\times X}\varphi\left(\frac{\psi(d(x,y))}r\right)\,dq_{123}(x,y,z)+
\frac st\int_{X\times X\times X}\varphi\left(\frac{\psi(d(y,z))}s\right)\,dq_{123}(x,y,z)\\
&\le&\frac rt\cdot 1+\frac st\cdot 1=1.
\end{eqnarray*}
Hence, $W_\vartheta(\mu_1,\mu_3)\le t$. This proves the triangle inequality.

\medskip

To prove completeness, assume that $(\mu_k)_k$ is a $W_\vartheta$-Cauchy sequence, say $W_\vartheta(\mu_n,\mu_k)\le t_n$ for all $k\ge n$ with $t_n\to0$ as $n\to\infty$. Then there exist couplings $q_{n,k}$ of $\mu_n$ and $\mu_k$ s.t.
\begin{equation}\label{cauchy}
\int\varphi\left(\frac1{t_n}\psi(d(x,y))\right)\,dq_{n,k}(x,y)\le1.\end{equation}
Jensen's inequality implies
$$\int \tilde d(x,y)\,dq_{n,k}(x,y)\le t_n\cdot\varphi^{-1}(1)$$
with $\tilde d(x,y):=\psi(d(x,y))$. The latter is a complete metric on $X$ with the same topology as $d$.
That is, $(\mu_k)_k$ is a Cauchy sequence w.r.t. the $L^1$-Wasserstein distance on $\mathcal P(X,\tilde d)$.
Because of completeness of $\mathcal P_1(X,\tilde d)$, we thus obtian an accumulation point $\mu$ and a converging subsequence $({\mu_k}_i)_i$. According to \cite{Villani2}, Lemma 4.4, this also yields an accumulation point $q_n$ of the sequence $(q_{n,k_i})_i$.
Continuity of the involved cost functions -- together with Fatou's lemma -- allows to pass to the limit in
(\ref{cauchy})
to derive
$$\int\varphi\left(\frac1{t_n}\psi(d(x,y))\right)\,dq_{n}(x,y)\le1$$
which proves that $W_\vartheta(\mu,\mu_n)\le t_n\to0$ as $n\to\infty$.

\medskip

With a similar argument, one verifies that $W_\vartheta(\mu,\nu)=0$ if and only if $\mu=\nu$.
\end{proof}

\begin{remark}
For each pair of probability measures $\mu,\nu$ on $X$
$$W_\vartheta(\mu,\nu)\le 1\quad\Longleftrightarrow\quad
\inf_{q\in\Pi(\mu,\nu)}\int_{X\times X} \vartheta(d(x,y))\, dq(x,y)\le 1.$$
\end{remark}

\end{document}